\documentclass{amsart}

\usepackage{amsmath}
\usepackage{amsfonts}
\usepackage{amssymb}
\usepackage{amstext}
\usepackage{amsbsy}
\usepackage{amsopn}
\usepackage{amsthm}
\usepackage{amsxtra}
\usepackage{graphicx}
\usepackage{caption}
\usepackage{subcaption}
\usepackage{color}
\usepackage{hyperref}
\usepackage{enumerate}
\usepackage{amscd}
\usepackage{bm}
\usepackage{enumitem}
\usepackage{cleveref}

\newlist{types}{enumerate}{10}
\setlist[types]{label*=\arabic*}

\crefname{typesi}{type}{types}
\Crefname{typesi}{Types}{types}

\newtheorem{theorem}{Theorem}[section]
\newtheorem{lemma}[theorem]{Lemma}
\newtheorem{corollary}[theorem]{Corollary}
\newtheorem{proposition}[theorem]{Proposition}

\newtheorem*{conjecture*}{Conjecture}
\newtheorem*{claim*}{Claim}
\newtheorem*{theorem*}{Theorem}

\theoremstyle{remark}
\newtheorem{remark}[theorem]{Remark}
\theoremstyle{definition}

\newtheorem{example}[theorem]{Example}

\newcommand{\A}{\mathcal{A}}

\newcommand{\ZZ}{\mathbb{Z}^2}

\newcommand{\Z}{\mathbb{Z}}

\newcommand{\N}{\mathbb{N}}

\newcommand{\Aut}{{\rm Aut}}
\newcommand{\Prob}{{\rm Prob}}

\begin{document}

\title[Normal amenable subgroups]{\Small Normal amenable subgroups of the automorphism group of sofic shifts}
\author{Kitty Yang} \thanks{This research was supported in part by the National Science Foundation grant 1502632, ''RTG: Analysis on manifolds", at Northwestern University.}
\address{Department of Mathematics, Northwestern University, Evanston, IL 60208 USA}
\email{kyang@math.northwestern.edu}

\begin{abstract}
Let $(X, \sigma)$ be a transitive sofic shift and let $\Aut(X)$ denote its automorphism group. We generalize a result of Frisch, Schlank, and Tamuz to show that any normal amenable subgroup of $\Aut(X)$ must be contained in the subgroup generated by the shift. We also show that the result does not extend to higher dimensions by giving an example of a two-dimensional mixing shift of finite type due to Hochman whose automorphism group is amenable and not generated by the shift maps.
\end{abstract}

\maketitle

\section{Introduction}
\subsection{The center of the group of automorphisms of sofic shifts}
Fix a finite alphabet $\A$. We define a subshift to be a closed set $X \subset \A^\Z$ that is invariant under the shift map $\sigma \colon \A^\Z \to \A^\Z$, where $(\sigma x)_n = x_{n+1}$, for $(x_n \colon n \in \Z) \in X$. A well-studied class are shifts of finite type, or SFTs, which can be described by a finite set of forbidden words (these notions are defined precisely in Section \ref{sec:background}). Closing SFTs under passage to factors, we obtain the sofic shifts, and this is the context of the result in this paper. There has been much interest in the algebraic properties of the automorphism group of the full shift and mixing SFTs (e.g. \cite{BLR,HED, KR,R1}), and many of these results generalize to sofic shifts. The automorphism group $\Aut(X)$ is the group of homeomorphisms of $X$ to itself that commute with $\sigma$, under composition. For any shift $X$, $\Aut(X)$ trivially contains $\langle \sigma \rangle$, the subgroup generated by the shift. By the Curtis-Hedlund-Lyndon Theorem \cite{HED}, any automorphism is a block code; as a corollary, $\Aut(X)$ is always a countable group.

Many results regarding automorphism groups of SFTs are obtained by manipulating marker automorphisms, initially introduced in \cite{HED} as a class of finite order automorphisms of the full shift. Using marker automorphisms, Hedlund proved that the automorphism group of the full shift contains, among others, isomorphic copies of every finite group and the free group on two generators. Complications arise for mixing SFTs with the existence of forbidden patterns, but Boyle, Lind, and Rudolph \cite{BLR} showed that marker automorphisms can also be defined for SFTs, and generalized Hedlund's results to SFTs.

Also using marker automorphisms, Kim and Roush \cite{KR} embedded the automorphism group of the full shift into the automorphism group of any mixing SFT, using markers to encode data words to act as symbols in a full shift. As a corollary, the automorphism groups of the full two-shift and the full three-shift contain the same subgroups up to isomorphism, but it remains an open question whether these automorphism groups are isomorphic. We note that the Kim and Roush construction does not embed the automorphism group of a mixing SFT into the automorphism group of another mixing SFT, as the embedding relies heavily on the lack of forbidden words in the full shift.

On the other hand, a method to distinguish automorphism groups follows from Ryan's Theorem \cite{R1,R2}. By definition, $\langle \sigma \rangle$ is contained in the center of $\Aut(X, \sigma)$. An SFT is transitive if it contains a point whose orbit is dense. Ryan showed that for any transitive SFT, the center is the subgroup generated by the shift. In the full four-shift, the shift map has a root, while the shift map for the full two-shift does not. Using Ryan's Theorem, Boyle, Lind, and Rudolph \cite{BLR} observed that these automorphism groups cannot be isomorphic, as the automorphism group of the full four-shift contains an element not in the center whose square is in the center, while the automorphism group of the full two-shift has no such element.

A topological factor of an SFT is not necessarily an SFT. To address this Weiss \cite{WEI} introduced the notion of sofic shifts, which form the smallest class of subshifts that contain SFTs and is closed under taking factors. It is natural to ask whether results about automorphism groups of SFTs also apply to sofic shifts. As the definition of marker automorphism relies on the fact that all sufficiently long words are synchronizing and transitive sofic shifts contain an abundance of synchronizing words, many of the results about subgroups of automorphism groups of SFTs also carry over to sofic shifts.

More recently, Frisch, Schlank, and Tamuz \cite{FST} generalized Ryan's Theorem to show that any normal amenable subgroup of the automorphism group of the full shift must be contained in the subgroup generated by the shift (see Section \ref{subsec:boundaries} for precise definitions). In this paper, we extend their result to any transitive sofic shift:
\begin{theorem} \label{theorem:thm1}
Let $(X, \sigma)$ be a transitive sofic shift. Any normal amenable subgroup of $\Aut(X)$ is contained in $\langle \sigma \rangle$.
\end{theorem}
This generalizes Ryan's Theorem, as subgroups of the center are always normal amenable subgroups, and in the case of a transitive sofic shift, our result says the converse also holds.

Complications arise when working in sofic shifts, as there can be arbitrarily long non-synchronizing words. Such words are necessary for marker automorphisms to actually be an automorphism of $X$. To overcome this problem, we first construct marker automorphisms for left-periodic points composed of synchronizing words as in the definition of extreme proximality (see Section \ref{subsec:boundaries}). We then extend the result to non-synchronizing left-periodic points that are in the $\rm{Aut}(X)$-orbit of the synchronizing points.

\subsection{Methods of proof}
As in \cite{FST}, we use a characterization of the topological boundary due to Furman \cite{FUR} to prove Theorem \ref{theorem:thm1} (see Section \ref{sec:background} for precise definitions).

For any shift $X$, $\Aut(X)$ acts on the set of left-periodic points. In the case of a transitive sofic shift, we show that an invariant subset of this space is in fact a topological boundary for $\rm{Aut}(X)$, and the kernel of the action is $\langle \sigma \rangle$. By Furman's characterization, we can conclude that any normal amenable subgroup of $\Aut(X)$ must be contained in $\langle \sigma \rangle$.

A key ingredient of the paper is Proposition \ref{proposition:prop1}: given any two cylinder sets of the boundary, we construct an automorphism of $X$ that maps one into the other. This is not needed in the case of the full shift, as \cite{FST} explicitly constructs a sequence of automorphisms which map proper closed subsets of the boundary to a particular point.

In the proof of Theorem \ref{theorem:thm1}, we show that the action of $\Aut(X)$ on the boundary is extremely proximal. This gives an alternate proof that for an infinite transitive sofic shift, the automorphism group contains a copy of the free group on two generators, and more generally, the free group on any number of generators. 

\subsection{Higher dimensions}
In Section \ref{sec:higher dim} we highlight the obstructions that arise when adapting these methods to higher dimensional shifts. Hochman \cite{HOC} proves a higher-dimensional analogue of Ryan's Theorem: for a transitive $\Z^d$ SFT with positive entropy, the center of the automorphism group is the subgroup generated by the shifts, which can be naturally identified with $\Z^d$. While \cite{FST} also shows that for the full $\Z^d$-shift, any normal amenable subgroup of the automorphism group is contained in the subgroup generated by the shift maps, $\Z^d$, there is an example due to Hochman of a topological mixing $\ZZ$ SFT with positive entropy whose automorphism group is amenable but strictly larger than $\ZZ$.

\section{Preliminaries} \label{sec:background}
\subsection{Shift systems}
Let $\A$ be a finite alphabet endowed with the discrete topology and equip $\A^\Z$ with the product topology. For $x = (x_n \colon n \in \Z) \in \A^\Z$, let $x_n \in \A$ denote the value of $x$ at $n \in \Z$. Define the shift map $\sigma \colon \A^\Z \to \A^\Z$ by $\sigma x_n = x_{n+1}$ for any $x \in \A^\Z$. If $X \subset \A^\Z$ is a closed and shift-invariant subset, we call $(X, \sigma)$ a \emph{shift}. To avoid trivial cases, we assume $X$ is infinite. Given $x \in X$, let
$$\mathcal{O}(x) := \{\sigma^i x: i \in \Z\}$$
be the orbit of $x$ under the shift, and $\overline{\mathcal{O}(x)}$ denote its closure in $X$.

Given an interval $[i, i+n-1] \subset \Z$, let $x_{[i,i+n-1]}$ be the word $w$ in $\A^{n}$ given by $w_j=x_{i+j}$ for $j = 0, 1, \dots, n-1$. A word $w$ in $\A^n$ is \emph{allowable in X} if there exists $x \in X$ and $j \in \Z$ such that $w=x_{[i,i+n-1]}$; we say that $w$ \emph{occurs in x at i}. For any word $w \in \A^{n}$, let $|w|=n$ be the length of the word. We denote the collection of allowable words of length $n$ in $X$ by $\mathcal{L}_n(X)$, and the \emph{language of $X$}, $\mathcal{L}(X) = \bigcup_{n \in \N} \mathcal{L}_n(X)$, is the set of all finite words that occur in $X$. Given two words, $u$ and $w$, $uw$ is the word in $\A^{|u|+|w|}$ obtained by concatenating $u$ and $w$; when we concatenate a word with itself, we will simplify by writing $w^2$.

Given a word $w \in \mathcal{L}(X)$, define the cylinder set $[w] \subset X$ to be
$$[w] := \{x \in X: x_n = w_n \text{ for } 0 \le n < |w|\}.$$
Such cylinder sets are clopen and, together with their translates, form a basis for the subspace topology on $X$. Thus, we can describe topological properties of a shift in terms of its language. A shift $(X, \sigma)$ is \emph{transitive} if for any pair of words $u$ and $w \in \mathcal{L}(X)$, there is some word $v$ such that $uvw\in \mathcal{L}(X)$; a shift $(X, \sigma)$ is \emph{mixing} if for any $ u,w\in \mathcal{L}(X)$, there exists an $N$ such that for any $n \ge N$, there is a word $v \in \mathcal{L}_n(X)$ such that $uvw$ is again allowable. Note that mixing implies transitivity.

A word $w \in \mathcal{L}(X)$ is \emph{synchronizing} if whenever $uw$ and $wv \in \mathcal{L}(X)$, then $uwv$ is again allowable in $X$. It follows that if $w$ is synchronizing, then any word that contains $w$ must also be synchronizing.

A point $x \in X$ is \emph{periodic} if there exists $k \in \N$ such that $x_n = x_{n-k}$ for all $n \in \Z$; we say $x \in X$ is \emph{left-periodic up to $N \in \Z$} if there exists $k \in \N$ such that $x_n = x_{n-k}$ for all $n < N$ and $x_N \ne x_{N-k}$. Note that a left-periodic point is not periodic, and that the index $N$ is independent of the choice of $k$. When the periodic index $k$ is minimal, we say that $x$ is (left-)$k$-periodic. Denote the set of points in $X$ of period $k$ by $\rm{Per}_k$. If $X$ contains a left-periodic point, then it must also contain a periodic point, but the converse is not true.

A shift is of \emph{finite type}, or an SFT, if it can be described by a finite set of forbidden words; that is, $X$ is an SFT if there exists a finite set of words $\mathcal{F}$ such that $x \in X$ if and only if any word that occurs in $x$ is not an element of $\mathcal{F}$. An SFT is \emph{$j$-step} if $\mathcal{F}$ consists of words of length $j+1$. We can also characterize SFTs using synchronizing words: a shift $X$ is an SFT if all words of sufficient length are synchronizing.

When $(X, \sigma)$ is a transitive SFT, it is a classical result that $X$ can be decomposed into disjoint mixing components which are cyclically permuted. More precisely, there exists a period $p$ and subsets $\{E_i\}_{1 \le i \le j}$ such that each $(E_i,\sigma^p)$ is mixing, and $\sigma E_i = E_{i+1 \mod j}$ (see \cite[\textsection 3]{AM}\cite[p.~543]{BOY}). Here, $p$ refers to the greatest common divisor of all $k$ with $\rm{Per}_k \ne \emptyset$. This extra structure says that in the case of a transitive SFT, given $w$ and $u$, we can extend $u$ on the left to $\widetilde{u}$ and for sufficiently large $n$, there is a word $v$ of length $np$ with $wv\widetilde{u} \in X$. Transitivity also implies that periodic points are dense for SFTs.

%Since any forbidden set $\mathcal{F}$ has a word of maximum length, any SFT is an $M$-step SFT for some $M$. If $X$ is an $M$-step SFT, $X$ is also an $N$-step SFT for any $N \ge M$.

The class of SFTs is not closed under factors, and we can consider the larger natural class which is. Given a compact metric space $X$ and a homeomorphism $T$, a \emph{topological factor} is a system $(Y, S)$ with a continuous surjective map $\pi:X \to Y$ satisfying $\pi \circ T = S \circ \pi$. We say that a shift is \emph{sofic} if it is a topological factor of an SFT. There are many equivalent definitions of sofic shifts, and we refer the reader to \cite[Theorem 3.2.1]{LM} for more details. Sofic shifts are the smallest class of shifts which are closed under taking factors and contain SFTs. Note that transitivity and mixing are each preserved under factors, and a transitive (cf. mixing) sofic shift is a factor of a transitive (cf. mixing) SFT. It follows that in a transitive sofic shift, periodic points are also dense. By recoding, we can assume that this factor map is a 0-block map. This presentation is convenient as it allows us to lift words in the sofic shift to words of the same length in the SFT. It follows that in transitive sofic shifts, as with transitive SFTs, between any two allowable words, we can insert arbitrary spacer words whose lengths form an arithmetic progression.

\subsection{Automorphisms of a shift}

An \emph{automorphism of $(X, \sigma)$} is a homeomorphism from $X$ to itself that commutes with the shift map.

By the Curtis-Hedlund-Lyndon Theorem \cite{HED}, any automorphism is defined by a \emph{block code}: given an automorphism $g \in \Aut(X)$, there exists an $R \in \N$ and a map $\hat{g}\colon \mathcal{L}_{2R+1}(X) \to \A$ such that $g x_i = \hat{g}(x_{i-R,\cdots, i+R})$. We say that $R$ is a \emph{range} for $\varphi$.

The set of automorphisms of $X$ under composition forms a group $\Aut(X, \sigma)$, or simply $\Aut(X)$ when $\sigma$ is clear from context. Since only finitely many automorphisms can have a given range, $\Aut(X)$ is countable. Given two automorphisms $g_1, g_2 \in \Aut(X)$, let $g_1g_2$ denote the composition $g_1 \circ g_2$.

In general, it is difficult to construct automorphisms of an arbitrary shift; however, if a shift contains synchronizing words, there are finite order automorphisms called \emph{marker automorphisms} originally defined by Hedlund \cite{HED} for full shifts, and later for SFTs by Boyle, Lind, and Rudolph \cite{BLR}. We now define marker automorphisms more generally, making slight modifications to conventions introduced in \cite{FST}.

We say that two words $w$ and $u$ \emph{overlap} if we can write $w=w'v$ and $u = vu'$ (or vice versa). When needed, we specify the length of overlap, and we say that $w$ and $u$ overlap with length $i$, where $i = |v|$.

Let $(X, \sigma)$ be a shift and $M_\ell$ and $M_r \in \mathcal{L}(X)$ be synchronizing words. Let $\mathcal{D} \subset \mathcal{L}_n$ be a set of words of length $n$ appearing in $X$ such that words of the form $M_\ell dM_r$ are allowable for all $d \in \mathcal{D}$. Suppose these words satisfy the following overlap condition: for any $d$ and $d' \in \mathcal{D}$, if $M_\ell dM_r$ and $M_\ell d'M_r$ overlap nontrivially with length $i$, then $i \le \min(|M_\ell|, |M_r|)$. Then any permutation $\tau$ of $\mathcal{D}$ induces an automorphism $g_\tau$ on $X$ by sending words of the form $M_\ell dM_r$ to $M_\ell \tau(d)M_r$ and leaving other words unchanged. Such an automorphism is called a \emph{marker automorphism}, and we refer to $M_\ell$ and $M_r$ as the \emph{left and right markers}, respectively, and $d \in \mathcal{D}$ as \emph{data words}. We note that as originally defined for a $j$-step SFT, the length of marker words have to be greater than $j$. The key is that such words are synchronizing, which is the necessary condition to ensure that applying the map does not introduce forbidden words.

\begin{example}
Let $X \subset \{0,1\}^\Z$ be the 1-step SFT defined by the forbidden word $11$. $(X, \sigma)$ is the golden mean shift.

Let $M_\ell = 100$ and $M_r = 0101$ be start and end markers, and $\mathcal{D} = \{0, 1\}$ be data words.
Given special blocks of the form $M_\ell dM_r$ and $M_\ell d'M_r$, for $d, d' \in \mathcal{D}$, $M_\ell dM_r$ and $M_\ell d'M_r$ can only overlap nontrivially by length 1.

Let $g \in \Aut(X)$ be the marker automorphism induced by the nontrivial permutation on $\mathcal{D}$. It permutes blocks of the form
\begin{align*}
\cdots 100 &\textbf{1} 0101 \cdots \\
\cdots 100 &\textbf{0} 0101 \cdots
\end{align*}
and leaves other blocks unchanged.
\end{example}

\begin{example}
Let $(Y, \sigma) \subset \{0,1\}^\Z$ be the even shift, consisting of bi-infinite sequences with only even number of consecutive 1s. $Y$ is a factor of the golden mean shift in the previous example, so it is sofic, but it is not an SFT. The word $1^{2i+1}$ for any $i \in \mathbb{N}$ is allowable (as a subword of $01^{2i+2}0$) and $01^{2i+1}, 1^{2i+1}0$ are also allowable. However, $01^{2i+1}0$ is not in the language of $Y$, so $1^{2^i+1}$ is not a synchronizing word. We have generated arbitrarily long non-synchronizing words; thus, $Y$ cannot be an SFT. We call such shifts \emph{strictly sofic}.

We note that any word which contains $0$ is a synchronizing word, so we can define marker automorphisms with markers that contain $0$.
\end{example}

\subsection{Generalized Ryan's Theorem for sofic shifts}
By generalizing the definition of marker automorphisms, we can adapt the proof of Ryan's Theorem to show that for a transitive sofic shift, the center of the automorphism group must be the subgroup generated by the shift. We show that for a transitive sofic shift, the automorphism group contains enough markers so that Ryan's Theorem still holds. The key proposition needed, which we state without proof, is that a transitive shift contains infinitely many synchronizing words

\begin{proposition}\cite[Proposition 3.3.16]{LM} \label{prop:synchronizing}
Suppose $(X, \sigma)$ is a transitive sofic shift. Then any word $w \in \mathcal{L}(X)$ can be extended on the right to a synchronizing word $wu$.
\end{proposition}

\begin{remark}
We note that in \cite{LM}, sofic shifts are defined as the set of all bi-infinite paths on a labeled graph. The definition of synchronizing word in \cite{LM} is dependent on the graph, while they use the term \emph{intrinsically synchronizing} to denote words we call synchronizing. However, if one chooses the minimal graph presentation for the sofic shift, these definitions coincide.
\end{remark}

To prove Ryan's Theorem for transitive sofic shifts, it suffices to show that there exist infinitely many synchronizing words which do not overlap themselves.

\begin{lemma} \label{lem:syncmarkers}
Let $(X, \sigma)$ be a transitive sofic shift. Then for any $n \in \mathbb{N}$, there is a synchronizing word $M$ of at least length $n$ which does not overlap itself non-trivially.
\end{lemma}
\begin{proof}
%As periodic points are dense, we can find a point $x \in X$ of period $n$, where $n$ is arbitrarily large. Consider the word $x_{[0,n-1]}$. If this word is synchronizing, then we are done. If not, by Proposition \ref{prop:synchronizing}, we can extend it to a word $x_{[0,1-1]}u$, which is synchronizing. Again, by density of periodic points, this must be a subword of some periodic point $y \in X$ of period $N>n$. Then the word $w = y_{[0,N-1]}$ is a synchronizing word which cannot overlap itself, else $y$ would have period less than $N$.
By the definition of periodicity, if $x \in X$ is $k$-periodic, then there must be some subword of length $k$ appearing in $x$ which does not overlap itself (otherwise $k$ would not be minimal).

Let $w$ be a synchronizing word of at least length $n$. By transitivity, there exists $u \in \mathcal{L}(X)$ such that $wuw \in \mathcal{L}(X)$. Note that $wuw$ is again a synchronizing word.

Since periodic points are dense, there is a periodic point $x$ of period $k \ge |wuw|$ such that $wuw$ appears in $x$. By the observation above, $x$ must contain a subword $M$ of length $k$ which does not overlap itself. Since $|wuw| \le k$, $w$ must appear in $M$, and thus $M$ is a synchronizing word.
\end{proof}

For completeness, we give a proof of the generalized Ryan's Theorem, due to Kitchens \cite[Theorem 3.3.22]{KIT}.

\begin{theorem} Let $(X, \sigma)$ be a transitive sofic shift. The center of $\rm{Aut}(X)$ is $\langle \sigma \rangle$.
\end{theorem}

\begin{proof} Let $(X, \sigma)$ be a transitive sofic shift, and let $\varphi \in \rm{Aut}(X)$ commute with all automorphisms. Suppose $\varphi$ has range $R$. Recall that for transitive sofic shifts, between any two words we can always insert spacers of lengths that form an arithmetic progression, where the difference is $p$, the period of $X$. Using these spacers and the sufficiently long markers produced by Lemma \ref{lem:syncmarkers}, we can find a synchronizing $M \in \mathcal{L}(X)$ and $n \in \mathbb{N}$, with $2R+1 \le n \le |M|$, such that for
$$\mathcal{D}(M,n): = \{d \in \mathcal{L}_n(X): MdM \in \mathcal{L}(X)\},$$
every word of length $2R+1$ appears as a subword of some element of $\mathcal{D}(M,n)$. This can be done by applying the transitive property simultaneously to $M$ and words of length $2R+1$ so that the spacers are of the same length. If necessary we can extend $M$ to the left. Repeat the process on the right to get words of the form $MdM$.
For any permutation $\tau \in \rm{Sym}(\mathcal{D}(M,n))$, let $g_\tau$ denote the marker automorphism induced by $\tau$.

Consider the periodic points of period $|M|+n$ obtained by concatenating $Md$ with itself, for any $d \in \mathcal{D}(M,n)$. We denote such points $\rm{Per}(M,n) \subset \rm{Per}_{|M|+n}$. Let $\rm{Orb}(M,n)$ be the set of distinct $\sigma$-orbits in $\rm{Per}(M,n)$. Note that $|\rm{Orb}(M,n)| \ge 2$, as each word of length $2R+1$ appears in some $d \in \mathcal{D}(M,n)$.

For any permutation of $\rm{Orb}(M,n)$, there is a $g_\tau$, for some $\tau \in \rm{Sym}(\mathcal{D}(M,n))$ whose action on $\rm{Orb}(M,n)$ coincides with the given permutation. In addition, $g_\tau$ acts as the identity on periodic points of period $|M|+n$ which are not in $\rm{Per}(M,n)$.

We claim that $\varphi$ acts on $\rm{Orb}(M,n)$. Suppose not. Then $\varphi$ maps some $x \in \rm{Per}(M,n)$ to a periodic point $y$ not in $\rm{Per}(M,n)$. Since $\varphi$ commutes with all $g_\tau$, this means that all points in $\rm{Per}(M,n)$ are mapped to the $\sigma$-orbit of $y$, which contradicts that $\varphi$ permutes the periodic points of each period.

Now we show that $\varphi$ acts as the identity permutation. Let $x$ and $y \in \rm{Per}(M,n)$ be in distinct $\sigma$-orbits. We note that there exists $\varphi(x) = \sigma^j(x)$ for some $-R \le j \le R$. This equality holds for all points in the $\sigma$-orbit of $x$, and we show that it holds for $y$ as well. Let $g_\tau$ be a permutation that takes $\mathcal{O}(x)$ to $\mathcal{O}(y)$. As $\varphi$ commutes with $g_\tau$, we have
$$\varphi(y) = g_\tau^{-1} \circ \varphi \circ g_\tau (y) = \sigma^j y.$$
As every block of length $2R+1$ appears in some $d \in \mathcal{D}(M,n)$, we conclude $\varphi=\sigma^j.$

%Let $x \in X$ be a sequence that contains a word of the form $Mu_iM$, where $d_i$ is centered at coordinate 0.
%\begin{claim} $\varphi x$ must contain a word of the form $Mu_jM$, where $d_j$ is centered at coordinate $-2R \le k \le 2R$.
%end{claim}

%Suppose not. Then for any $\tau \in \rm{Perm}(\mathcal{B}(M,n))$, we have 

%$$\varphi x (0) = g_\tau \varphi x(0) = \varphi g_\tau x(0) = \widehat{\varphi}(\tau(d_i))$$

%but $f$ cannot be constant, as it contradicts the bijectivity of $\varphi$. So, $\varphi x$ must contain a special block near coordinate 0. It cannot be more than $2R$ to the left or right of 0. Because the marker $M$ is unchanged in the special block, this would contradict the injectivity of $\varphi$. This proves the claim.

%\begin{claim} $x$ and $\varphi x$ contain the same word $Md_iM$ near coordinate 0.
%\end{claim}

%Again we show this by contradiction. Suppose $d_i \ne d_j$. Then we can choose a permutation $\tau$ that sends $d_j \mapsto u$ and $d_i \mapsto w$, where $\widehat{\varphi}(w) \ne u_{-\ell}$. Then

%$$\widehat{\varphi}(w) = \widehat{\varphi}(\tau(d_i) = \varphi g_\tau x(0); \text{ while } g_\tau \varphi x (0) = u_{-\ell}$$
%which contradicts the commutativity of $\varphi$ and $\gamma_\pi$.

%Since all words of length $2R+1$ appear in $\mathcal{B}(M,n)$, $\varphi = \sigma^{-\ell}$.
\end{proof}

\subsection{Topological boundaries} \label{subsec:boundaries}
Throughout this section, let $G$ be a locally compact group and let $\Omega$ be a compact metric space with a continuous $G$ action $G \times \Omega \to \Omega$: for any $g \in G$ and $\omega \in \Omega$,
$$(g, \omega) = g \cdot \omega.$$
We call $\Omega$ a \emph{$G$-space}. Given $\omega \in \Omega$, let $G\omega$ denote the $G$-orbit of $\omega$:
$$G\omega = \{g \cdot \omega: g \in G\} \subset \Omega$$
and $\overline{G\omega}$ its closure in $\Omega$.

Let $\Prob(\Omega)$ be the set of Borel probability measures on $\Omega$, equipped with the weak-* topology. Since $\Omega$ is compact, $\Prob(\Omega)$ is also a compact metric space. Given $\omega \in \Omega$, let $\delta_\omega$ denote the Dirac measure concentrated at $\omega$. The mapping $\omega \mapsto \delta_\omega$ gives an embedding of $\Omega$ into $\Prob(\Omega)$.

The $G$-action on $\Omega$ induces an action on $\Prob(\Omega)$ by viewing elements of $G$ as self-homeomorphisms of $\Omega$: for any $g \in G$ and $\omega \in \Omega$,
$$g \cdot \mu = \mu \circ g^{-1}.$$

We say that the $G$-action on $\Omega$ is \emph{minimal} if for any $\omega \in \Omega$, the $G$-orbit closure $\overline{G\omega} = \Omega$. The $G$-action on $\Omega$ is \emph{strongly proximal} if for all $\mu \in \Prob(\Omega)$, the $G$-orbit closure $\overline{G\mu} \subset \Prob(\Omega)$ contains a Dirac measure $\delta_\omega$ for some $\omega \in \Omega$. A $G$-space $\Omega$ is a \emph{topological boundary} if the $G$-action on $\Omega$ is minimal and strongly proximal.

The $G$-action on $\Omega$ is \emph{extremely proximal} if $|\Omega| \ge 2$ and for any proper closed set $C \subsetneq \Omega$ and any open set $U \subset \Omega$, there is some $g \in G$ with $gC \subset U$.

It is known that extreme proximality implies strong proximality \cite[\textsection 3]{G2} and the product of strongly proximal actions is again strongly proximal \cite[\textsection 3]{G1}.

A group $G$ is \emph{amenable} if for every compact $G$-space $\Omega$, the $G$-action on $\Prob(\Omega)$ has a fixed point. Examples of amenable groups include abelian groups and finite groups, while the free group is not amenable.

\section{Topological boundaries of the automorphism group of transitive sofic shifts}
\subsection{The action of the automorphism group on left-periodic points} \label{subsec:q_k}
Recall from Section \ref{sec:background} that left-periodic points are not periodic. For any shift $(X, \sigma)$, we define a compact space equipped with an $\rm{Aut}(X)$ action.
\begin{lemma} \label{lem:Q_k}
Let $(X, \sigma)$ be a shift and $k \in \N$. Suppose $X$ contains a left-$k$-periodic point. We denote the set of left-$k$-periodic points up to $k$ by $Q_k$. Then $Q_k$ is an $\Aut(X)$-space, and $\sigma$ acts trivially on $Q_k$. If $X$ contains a left-$k$-periodic point which is transitive, then the kernel of the action is $\langle \sigma \rangle.$
\end{lemma}

\begin{proof}
Since any automorphism is a block map, the set of left-$k$-periodic points is invariant under $\Aut(X)$. The set of all left-$k$-periodic points is precisely $\bigcup_{i \in \Z} \sigma^i Q_k$.
Thus, for any $g \in \Aut(X)$ and $x \in Q_k$,
\begin{equation}\label{eq:cocycle}
    gx \in \sigma^i Q_k,
\end{equation}
for some unique $i$, since the shifts of $Q_k$ are pairwise disjoint.
Define a cocycle $\alpha \colon \Aut(X) \times Q_k \to \Z$ to be:
\begin{equation}
    \alpha(g,x) = -i
\end{equation}
where $i$ is obtained from equation \eqref{eq:cocycle}.
The cocycle condition ensures that the induced $\Aut(X)$-action on $Q_k$ is well-defined, where for $g \in \Aut(X), x \in Q_k$:
\begin{equation}
    g \cdot x = \sigma^{\alpha(g, x)} \circ gx.
\end{equation}

We note here that the action of $\rm{Aut}(X)$ on $X$ is different from the action on $Q_k$, and use different notation to make clear which action we are referencing. For any $x \in Q_k$, $\alpha(\sigma,x) = -1$, so $\sigma \cdot x = x$.

Suppose in addition $x \in X$ is a transitive left-$k$-periodic point. Let $g \notin \langle \sigma \rangle$. For each $n \in \N$, let $R_n$ denote the maximum of $n$ and the range of $g$. Let $\hat{g}_n$ and $\hat{\sigma}^n$ denote the block codes of width $R_n$ that induce $g$ and $\sigma^n$, respectively. Then there exists a word $w_n$ of length $2R_n+1$ such that $\hat{g}_n (w_n) \ne \hat{\sigma}^n(w_n)$. Since every $w_n$ appears in $x$, $g \cdot x \ne x$.
\end{proof}

The set of $k$-periodic points $\rm{Per}_k$ is invariant under $\Aut(X)$. We can decompose $\rm{Per}_k$ into a disjoint union of distinct $\sigma$-orbits:
\begin{equation}\label{eq:Per_k}
    \rm{Per}_k = \mathcal{O}(x^1) \amalg \cdots \amalg \mathcal{O}(x^j).
\end{equation}
Thus, the action of $\Aut(X)$ on $X$ descends to an action on $\rm{Per}_k/\langle \sigma \rangle$.

\begin{lemma}
Let $(X, \sigma)$ be a shift that contains a left-$k$-periodic point, and $Q_k$ be the set of left-$k$-periodic points up to $k$. There exists a projection
\begin{equation}
\pi \colon Q_k \to \rm{Per}_k/\langle \sigma \rangle
\end{equation}
which is $\Aut(X)$-equivariant.
\end{lemma}

\begin{proof}
Given $x \in Q_k$, there exists a unique $k$-periodic point such that $y_n=x_n$ for all $n<k$. Define the projection $\pi$ which sends $x$ to the $\sigma$-orbit in $\rm{Per}_k$ containing $y$.
Since any automorphism is a block code, for any $x \in Q_k$ and $g \in \Aut(X)$, $\pi(gx) = g \pi (x)$.
\end{proof}

A map $s \colon \rm{Per}_k/\langle \sigma \rangle  \to Q_k$ is a section of the projection $\pi$ if $s$ is a right inverse of $\pi$. Let $\Omega$ be the collection of all sections $s \colon \rm{Per}_k \to Q_k$ of the projection $\pi$. Since $\pi$ is equivariant, the action of $\Aut(X)$ on $Q_k$ induces an action on $\Omega$: for any $g \in G$ and $s \in \Omega$,
$$gs = g \cdot (s \circ g^{-1})$$\label{par:omega}
where we view $g^{-1}$ as the permutation of $\rm{Per}_k/\langle \sigma \rangle$ induced by $g^{-1}$.

Let $\{\Omega^m\}$ denote the fibers of $\pi \colon \Omega \to \rm{Per}_k/\langle \sigma \rangle$. Given a periodic orbit $\mathcal{O}(x^m) \in \rm{Per}_k/\langle \sigma \rangle$,
\begin{equation}\label{eq:omega^m}
    \Omega^m = \{x \in \Omega:\exists i\in \Z \text{ with }x_n=x^m_{n-i} \text{ for all } n<k\}
\end{equation}

Let $N \triangleleft \Aut(X)$ be the normal subgroup given by the kernel of $\pi$. Since $N$ preserves the fibers $\{\Omega^m\}$, the restriction of the action on $\Omega$ to $N$ is isomorphic to the diagonal action of $N$ on the product of the fibers $\prod_{m=1}^j \Omega^m$, where $j$ is the number of distinct $k$-periodic orbits defined in \eqref{eq:Per_k}.

\subsection{Extreme Proximality} \label{subsec:extreme_proximality}
Given a bi-infinite sequence $x \in X$, we say that $x$ is \emph{synchronizing} if all sufficiently long words that appear in $x$ are synchronizing. If $x$ is periodic and some synchronizing word appears in $x$, then $x$ itself is synchronizing. Since the definition only depends on words that appear in $x$, a sequence is synchronizing if and only if any point in its orbit closure is synchronizing.

Every periodic orbit $\mathcal{O}(x)$ is exactly one of the three following types:
\begin{types}
    \item $x$ is synchronizing \label{type:type1}
    \item $x$ is not synchronizing, but there exists an automorphism $h \in \rm{Aut}(X)$ such that $hx$ is synchronizing \label{type:type2}
    \item $x$ is not synchronizing, and for all automorphisms $h \in \rm{Aut}(X)$, $hx$ is not synchronizing. \label{type:type3}
\end{types}

\begin{remark}
When $X$ is an SFT, all sequences are synchronizing and $\rm{Per}_k$ consists only of synchronizing points.
\end{remark}

In general, the action of $\rm{Aut}(X)$ on $\rm{Per}_k$ may not be transitive. In the case of SFTs, however, for sufficiently large $k$, $\rm{Aut}(X)$ does act on $\rm{Per}_k$ transitively (see \cite{BLR}). The proof constructs a composition of marker automorphisms which permute periodic points with disjoint orbits, building on work by Boyle and Krieger \cite{BK} for the full shift. The same proof shows that for a transitive sofic shift, $\rm{Aut}(X)$ acts transitively on the synchronizing points in $\rm{Per}_k$. However, non-synchronizing points do not contain   any synchronizing subwords, so they are fixed by all marker automorphisms.

Let $\rm{Syn}_k \subset \rm{Per}_k$ be the subset of periodic points of \cref{type:type1} and \cref{type:type2}. Then $\rm{Syn}_k$ is the $\Aut(X)$-orbit under $\Aut(X)$ of any synchronizing point, so $\Aut(X)$ must act transitively on $\rm{Syn}_k$.

\begin{remark}
It is possible that no periodic points of \cref{type:type2} exist. We do not know if there exist automorphisms which do not fix non-synchronizing points.
\end{remark}

Define $\widetilde{\Omega}$ analogously to $\Omega$ in Section \ref{subsec:q_k}, but for the restricted action. For any section $s \in \Omega$, where $s:\rm{Per}_k/\langle \sigma \rangle \to Q_k$, let $s|_{\rm{Syn}_k/\langle \sigma \rangle}: \rm{Syn}_k/\langle \sigma \rangle \to Q_k$ be the corresponding element in $\widetilde{\Omega}$. Note that this is not injective and many $s \in \Omega$ project to the same map in $\widetilde{\Omega}$.

We now show that $\widetilde{\Omega}$ is a topological boundary for $\rm{Aut}(X)$. Recall that $N$, the kernel of the action of $\rm{Per}_k$, acts on each $\Omega^{m}$. If we consider the action of $N$ on $\widetilde{\Omega}$,

$$\widetilde{\Omega} \cong \prod_{m} \Omega^m$$

where the product is taken over values of $m$ where $x^m$ is of \cref{type:type1} or \cref{type:type2}.

We show that the action of $N$ on $\Omega^m$ is extremely proximal, and use this to prove that the full action of $\rm{Aut}(X)$ on $\widetilde{\Omega}$ is a topological boundary. The key step is constructing marker automorphisms in $\Omega^m$, where $x^m$ is a synchronizing point. Then for non-synchronizing points, we exploit the fact that such points are in the $\Aut(X)$-orbit of some synchronizing point to achieve the same result.

The set $\Omega^m$ is closed in $X$, so cylinder sets of the form 
\begin{equation}
    [w]^m:= [w] \cap \Omega^m, \text{ where } w \in \mathcal{L}(X)
\end{equation}
form a subbase that generates the subspace topology on $\Omega^m$.

\begin{proposition} \label{proposition:prop1}
Let $(X, \sigma)$ be a transitive sofic shift and $k \in \N$. Suppose $X$ contains a left-$k$-periodic point. Fix $x^m \in \rm{Syn}_k$, and define $\Omega^m$ as in \eqref{eq:omega^m}.
Let $w,u$ be words in $\mathcal{L}(X)$ such that the corresponding cylinder sets $[w]^m$ and $[u]^m$ in $\Omega^m$ are nonempty and proper. Then there is an automorphism $g \in \Aut(X)$ which acts as the identity on $\rm{Per}_k/\langle \sigma \rangle$ and satisfies $g \cdot [w]^m \subset [u]^m$.
\end{proposition}

\begin{proof}
We fix $x^m \in \rm{Syn}_k$, so either $x^m$ is of \cref{type:type1} or \cref{type:type2}.\\
\textbf{Case 1:} $x^m$ is synchronizing. \\
There exists $c \in \mathbb{N}$ such that all subwords of length at least $c$ are synchronizing. 

Since points in $\Omega^m$ are left-$k$-periodic up to $k$, if $|w| \le k$, we can replace it with a word $\widetilde{w}$ of length $k+1$ such that $[w]^m = [\widetilde{w}]^m$. We may assume that $|w|,|u| >k$. Write $a = w_{[0, \dots, k-1]}$. If $w[0] \ne u[0]$, let $\widetilde u$ be the unique extension of $u$ so that $w$ and $\widetilde u$ begin with the same letter.

We can assume that $\widetilde{u}$ does not appear as the initial word of $w$, otherwise $u = \widetilde{u}$ and the identity automorphism satisfies the conclusion of lemma. We first deal with the case that $w$ does not appear as the initial word of $\widetilde{u}$.

Let $a^r$ be the word obtained by concatenating $r$ copies of $a$. Since it is a word of at least length $c$ which appears in $x^m$, it must be synchronizing. As $[w]^m \ne \emptyset$, $a^rw$ is allowable, and must also be synchronizing. Since $X$ is transitive, choose $v \in \mathcal{L}(X)$ such that $\widetilde uva^r$ is an allowable word. Since $a$ and $\widetilde u$ begin with the same initial word, we can choose $v$ with $|\widetilde uv| = rpk$ for some $r \ge \max\{|w|,c\}$, where $p$ is the greatest common divisor of $k$ with $\rm{Per}_k \ne \emptyset$. Set $a^r$ to be the left marker, $a^rw$ to be the right marker, and $\mathcal{D} = \{a^{rp}, \widetilde uv\}$ to be data words. 

To show these markers induce a well-defined marker automorphism, it suffices to check that special words of the form $a^rda^rw$, for $a^rd \in \mathcal{D}$, satisfy the overlap condition given in the definition of marker automorphism. The word $a$ is a word of length $k$ which appears in $x^m$, a $k$-periodic point. Since $k$ is minimal, $a$ cannot overlap itself nontrivially.

The initial word of $w$ and $\widetilde u$ is $a$, while $a$ does not occur at position $k$ in $w$ or $\widetilde u$. Thus, if the special words $a^ra^{rp}a^rw$ and $a^r\widetilde uva^rw$ overlap nontrivially, the length of the overlap must be either less than $|w|$, or exactly $rk+|w|$. In the second case, the special blocks would overlap by $a^r$. However, this overlap would force $\widetilde{u}$ to be the initial word of $w$, which contradicts the assumption.

Because special words begin with $a^r$, a similar argument shows that a special word can only overlap with itself nontrivially by at most $|w|$. By the choice of $r$, $|a^r| \ge |w|$, so the marker automorphism $g$ determined by the nontrivial permutation on $\mathcal{D}$ is well-defined.

Let $x^i \in \rm{Per}_k$. Since no special words appear, $g$ acts as the identity on $x^i$, so $g$ is in the kernel of the action on $\rm{Per}_k/\langle \sigma \rangle$.

Lastly, we show that $g \cdot [w]^m \subset [u]^m$. Let $y \in [w]^m$. Since 
$y$ is left-$k$-periodic, the first occurrence of a special word in $y$ is $a^ra^{rp}a^rw$ at $-(2r+rp)k$. Thus, 
$$gy = \cdots a\widetilde uv.w \cdots.$$
Applying the cocycle $\alpha$ gives 
$$g \cdot y = \cdots \widetilde a. uvw \cdots \in [u]^m$$
where $\widetilde a$ is the initial $k$-block of $u$.

Suppose now $w$ is the initial word of $\widetilde{u}$. We can partition $[w]^m$ by the finitely many allowable extensions of $w$ given by $wb$, where each $wb$ is of length $|\widetilde{u}|$. Applying the process above gives marker automorphisms $g_b$ for each extension $wb$. Since the end markers for each $g_b$ are distinct, they commute. The composition of $\{g_b\}$ is well-defined, and is a finite order automorphism that maps $[w]^m$ into $[u]^m$.

\textbf{Case 2:} $x^m$ is not synchronizing, and there is an $h \in \rm{Aut}(X)$ such that $hx^m$ is synchronizing. \\
Let $h \in \rm{Aut}(X)$ where $hx^m = x^i$ is synchronizing, and let $[w]^m,[u]^m$ satisfy the hypothesis. Consider the sets
$$h[w]^m \text{ and } h[u]^m.$$
We can partition them into finitely many cylinder sets, so by the previous construction, there is some $g \in N$ that maps
$$g \cdot h[w]^m \subset h[u]^m.$$
Then
$$h^{-1}gh \cdot [w]^m \subset [u]^m$$
with $h^{-1}gh \in N$, as desired.
\end{proof}

Recall that $N \triangleleft \Aut(X)$ is the kernel of $\pi:\Omega \to \rm{Per}_k/\langle \sigma \rangle$, so the automorphism produced above is contained in $N$.

\begin{corollary}\label{cor}
Let $(X, \sigma)$ be a transitive sofic shift and $k \in \N$. Suppose that $X$ contains a left-$k$-periodic point. Then the following hold:
\begin{enumerate}[label={(\arabic*)},ref={\thecorollary~(\arabic*)}]
\item The action of $N$ on $\Omega^m$, as defined in \eqref{eq:omega^m}, is minimal. \label{partone}
\item The action of $\Aut(X)$ on $\Omega$, as defined in Section \ref{par:omega} is minimal. \label{parttwo}
\item The $N$ action on $\Omega^m$ is extremely proximal. \label{partthree}
\item The $\Aut(X)$ action on $\Omega$ is strongly proximal. \label{partfour}
\end{enumerate}
\end{corollary}

\begin{proof}
(1) It suffices to show that the $N$-orbits of any nonempty open subset $U$ covers all of $\Omega^m$. Let $[u]^m \subset U$ be a nonempty cylinder set with $|u| \ge k$ and $[w]^m$ be an nonempty cylinder with $|w| \ge k$. By Proposition \ref{proposition:prop1}, there exists $g \in N$ such that $g \cdot [w]^m \subset [u]^m$. Since $g$ is defined by an inversion, $g = g^{-1}$, and so each  $[w]^m \subset g \cdot [u]^m$. As $[u]^m$ was arbitrary, this shows that $\bigcup_{g \in N} g \cdot U = \Omega^m$. \\
(2) Let $U \subset \Omega$ be an open set. If the intersections $U \cap \Omega^m$ are all nonempty, then by part (1), the $N$-orbit of $U$ covers $\Omega$.
Suppose $U$ is contained in some $\Omega^m$. The action of $\Aut(X)$ on $\rm{Syn}_k$ is transitive, so there exists $g \in \Aut(X)$ such that $gU \cap \Omega_n$ is nonempty for any $\Omega_n$. By part (1), the $\Aut(X)$-orbit of $U$ covers $ \Omega$.\\
(3) Each $\Omega^m$ contains more than two points, and cylinder sets form a subbase that generates the topology on $\Omega^m$. In addition, $\Omega^m$ is compact, so any closed set is covered by finitely many cylinder sets. By Proposition \ref{proposition:prop1}, the $N$ action on each $\Omega^m$ is extremely proximal. \\
(4) As extremely proximal actions are also strongly proximal, by part (3), the $N$ action on each $\Omega^m$ is strongly proximal. Thus, the product action of $N$ on $\prod_{m=1}^j \Omega^m$ is also strongly proximal. Since the diagonal action of $N$ on the product space is isomorphic (as continuous group actions) on $\Omega$, it follows that the action of $\Aut(X)$, which contains $N$, on $\Omega$ is also strongly proximal.
\end{proof}

We use a proposition of Furman which relates the kernel of boundary actions and normal amenable subgroups:

\begin{proposition}[Furman~\cite{FUR}]\label{proposition:furman}
Let $G$ be a discrete group, and consider the following subgroups of $G$:
\begin{enumerate}
\item $N = \bigcap_{i \in I} Ker(G \to Homeo(X_i))$, where $I$ is the set of isomorphism classes of boundary actions on the set of $G$-spaces,
\item $\sqrt{G}$, the group generated by all closed normal amenable subgroups in $G$.
\end{enumerate}

Then $N = \sqrt{G}$.
\end{proposition}

In particular, the kernel of any boundary action contains any normal amenable subgroup.

We have now assembled the ingredients to prove Theorem \ref{theorem:thm1}.

\begin{proof}[Proof of Theorem \ref{theorem:thm1}]
Let $(X, \sigma)$ be a transitive sofic shift. As periodic points are dense, there exists some $k$ such that $X$ contains $k$-periodic points and $\rm{Syn}_k \ne \emptyset$.  Since $X$ is not finite, $X$ also contains left-$k$-periodic points.

Corollaries \ref{parttwo} and \ref{partfour} show that $\Omega$ is an $\Aut(X)$-boundary. By Proposition \ref{proposition:furman}, any normal amenable subgroup of $\Aut(X)$ is contained in the kernel of a boundary action.

An element in $\Omega$ is a section of the projection $\pi:\Omega \to \rm{Per}_k/\langle \sigma \rangle$, so the kernel of $\Aut(X)$ acting on $\Omega$ must be contained in the kernel of $\Aut(X)$ acting on $Q_k$, the set of left-$k$-periodic points up to $k$. Thus, it follows from Lemma \ref{lem:Q_k} that the kernel of the $\Aut(X)$ action on $\Omega$ is precisely $\langle \sigma \rangle$, and we obtain the desired result.
\end{proof}

\section{Higher dimensions} \label{sec:higher dim}
We show that the direct analogue of Theorem \ref{theorem:thm1} in higher dimensions fails by giving a counterexample and explain why the methods of proof do not generalize even with stronger hypotheses. Consistent with the definition of one-dimensional shifts given in Section \ref{sec:background}, we define a $\Z^d$-shift to be a closed, translation-invariant subset of $\A^{\Z^d}$. A $\Z^d$-shift is an SFT if it can be described by forbidden patterns in $\A^\mathcal{F}$, for some finite set $\mathcal{F} \subset \Z^d$, and a $\Z^d$ sofic shift is a topological factor of a $\Z^d$ SFT. The automorphism group consists of self-homeomorphisms of the shift that commute with the shift maps, which can be identified with $\Z^d$.

Hochman \cite{HOC} constructs a two-dimensional SFT $X \subset \A^{\ZZ}$, which is topologically mixing and has positive entropy. Hochman explicitly computes the automorphism group to be $\ZZ \oplus \bigcup S_{i,j}$, where $\ZZ$ is generated by the shift maps and $\bigcup S_{i,j}$ is a directed union of infinitely many finite groups, arising from higher dimensional marker automorphisms. Amenability is closed under taking direct limits and sums; thus, the automorphism group is amenable. In higher dimensions, Ryan's Theorem holds \cite{HOC}, and the center is the subgroup generated by the shifts, $\ZZ$. In particular, $\Aut(X)$ has normal amenable subgroups that are not contained in the center. While this shift is topologically mixing, the set of periodic points is not dense, which suggests this may not be the right condition to impose.

There are various notions of uniform mixing in higher dimensions, (for example, strongly irreducible, uniform filling, and block gluing) each of which imply that periodic points are dense. In each case, if two allowable patterns are sufficiently far apart, there is another allowable pattern which agrees with the original patterns; the distinct notions of uniform mixing depends on the shape of patterns we consider. In contrast, for $d=1$, these definitions of uniform mixing are equivalent to topological mixing.

However, even with dense periodic points, we cannot construct a topological boundary for uniformly mixing $\Z^d$ SFTs as we did in the one-dimensional case. Because there are now more directions of periodicity, we cannot construct a space on which the automorphism group acts in the same manner. More specifically, we cannot define a $\Z^d$ cocycle as we did in equation \ref{eq:cocycle}.

In the case of the higher dimensional full shift, Frisch, Schlank, and Tamuz \cite{FST} show that any normal amenable subgroup must be contained in the subgroup generated by the shifts; unfortunately, their methods do not generalize to uniformly mixing SFTs. They construct a class of automorphisms of $\A^{\Z^d}$, induced by automorphisms of $\A^\Z$, which act independently on bi-infinite sequences of a configuration $x \in \A^{\Z^d}$. This relies strongly on the fact that in the full shift, there are no forbidden blocks. In a more general $\Z^d$-SFT, acting independently on lower dimensional subspaces may produce forbidden patterns. We note that higher dimensional marker automorphisms cannot arise from such a construction.

\bibliographystyle{abbrv}
\bibliography{normal-amenable}

\end{document}